\documentclass[12pt,a4paper,reqno]{amsart}
\usepackage[centertags]{amsmath}
\usepackage{amsfonts}
\usepackage{amssymb}

\usepackage{amsthm}
\usepackage{newlfont}
\usepackage{color}
\usepackage[margin=1in]{geometry}  % set the margins to 1in on all 
\usepackage{graphicx}              % to include figures
\usepackage{amsmath}               % great math stuff
\usepackage{amsfonts}              % for blackboard bold, etc
\usepackage{amsthm}   

\newtheorem{theorem}{Theorem}[section]
\newtheorem{proposition}[theorem]{Proposition}
\newtheorem{lemma}[theorem]{Lemma}
\newtheorem{corollary}[theorem]{Corollary}
\newcommand{\Tr}{\operatorname{Tr}}
\newcommand{\Dt}{\operatorname{det}}

\theoremstyle{definition}

\newtheorem{remark}[theorem]{Remark}
\theoremstyle{example}

\theoremstyle{problem}

\theoremstyle{property}

\begin{document}

\title[Some inequalities for the matrix Heron mean]{Some inequalities for the matrix Heron mean}

\author[Dinh Trung Hoa]{Dinh Trung Hoa}
\address{Division of Computational Mathematics and Engineering (CME), Institute for Computational Science (INCOS), Ton Duc Thang University, Ho Chi Minh City, Vietnam; \\ Faculty of Civil Engineering, Ton Duc Thang University, Vietnam.
}
\email{dinhtrunghoa@tdt.edu.vn}


\begin{abstract}
Let $A, B$ be positive definite matrices, $p=1, 2$ and $r\ge 0$. It is shown that   
\begin{equation*}
||A+ B + r(A\sharp_t B+A\sharp_{1-t} B)||_p \le ||A+ B + r(A^{t}B^{1-t} + A^{1-t}B^t)||_p.
\end{equation*}
We also prove that for positive definite matrices $A$ and $B$
\begin{equation*}\label{det}
\Dt (P_{t}(A, B)) \le \Dt (Q_{t}(A, B)),
\end{equation*}
where $Q_t(A, B)= \big(\frac{A^t+B^t}{2}\big)^{1/t}$ and $P_t(A, B)$ is the $t$-power mean of $A$ and $B$. As a consequence, we obtain the determinant inequality for the matrix Heron mean: for any positive definite matrices $A$ and $B,$
$$
\Dt(A+ B + 2(A\sharp B)) \le \Dt(A+ B + A^{1/2}B^{1/2} + A^{1/2}B^{1/2})).
$$
These results complement those obtained by Bhatia, Lim and Yamazaki (LAA, {\bf 501} (2016) 112-122).
\end{abstract}

\maketitle

%%%%%%%%%%%%%%%%%%%%%%%%%%%%%%%%%%%%%%%%%%%%%%%%%%%%%%%%%%%%%%%%%%%%%%%%%
\section{Introduction}
Let $M_n$ be the space of $n\times n$ complex matrices and $M_n^+$ the positive part of  $M_n$. Denote by $I$ the identity element of $M_n$. For self-adjoint matrices $A, B \in M_n$ the notation $A \le B$ means that $B - A \in  M_n^+$. For a real-valued function $f$ of a real variable and a self-adjoint matrix $A \in M_n$, the value $f(A)$ is understood by means of the functional calculus.

For $0\le t \le 1$ the $t$-geometric mean of $A$ and $B$ is defined as 
$$
A\sharp_t B = A^{1/2}(A^{-1/2}BA^{-1/2})^tA^{1/2}.
$$
The geometric mean $A\sharp B:= A\sharp_{1/2} B$ is the midpoint of the unique geodesic $A\sharp_t B$ connecting two points $A$ and $B$ in the Riemannian manifold of positive matrices.

Recently, Bhatia et al. \cite{bh} proved that for any positive definite matrices $A$ and $B$ and for $p=1, 2,$
\begin{equation}\label{bh1}
|| A+B + 2rA\sharp B||_p \le || A+B + r(A^{1/2}B^{1/2}+B^{1/2}A^{1/2})||_p.
\end{equation}
When $r=1$, inequality holds for $p= \infty$.

For the case $p=2$ the proof of (\ref{bh1}) is based on the following fact: for any positive definite $A$ and $B$,
\begin{equation}\label{bhatia1}
\lambda(A^{1/2}(A\sharp B)A^{1/2}) \prec_{\log} \lambda(A^{3/4}B^{1/2}A^{3/4}),
\end{equation}
where the notation $\lambda$ is used for the $n$-tuple of eigenvalues of a matrix $A$ in decent order and $\lambda(A) \prec_{\log} \lambda(B)$ means that 
$$
\prod_{j=1}^k \lambda_i(A) \le \prod_{j=1}^k \lambda_i(B), \quad 1 \le k\le n
$$ 
and inequality holds when $k=n$.

For $p=\infty$ inequality (\ref{bh1}) was proved by using a result of Lim and Yamazaki \cite[Theorem 4.1]{LY} 
\begin{equation}\label{ya}
||P_t(A_1, A_2,\cdots, A_n)||_\infty \le ||Q_t(A_1, A_2,\cdots, A_n)||_\infty,
\end{equation}
where {\em the power mean} $P_t(A_1, A_2,\cdots, A_n)$ of $A_1, A_2,\cdots, A_n$ \cite{LP} and is the unique solution of the matrix equation
$$
X = \frac{1}{m} \sum_{i=1}^m X\sharp_t A_i 
$$
and 
$$Q_t(A_1, A_2,\cdots, A_n) = \big(\frac{1}{m} \sum_{i=1}^m A_i^t\big)^{1/t}.$$
For $m=2$ Lim and P$\acute{a}$flia \cite[Remark 3.10]{LP} show that 
$$
P_t(A, B) = A \sharp_{1/t} \big(\frac{1}{2}(A+ A\sharp_t B)\big) = A^{1/2}\big(\frac{I+ (A^{-1/2}BA^{-1/2})^t}{2}\big)^{1/t}A^{1/2}.
$$
Hopefully, for $t=1/2$, 
$$
P_{1/2}(A, B) = \frac{1}{4}(A+B + A \sharp B), \quad Q_{1/2}(A, B) = \frac{1}{4}(A+B+A^{1/2}B^{1/2}+ B^{1/2}A^{1/2}).
$$
And so, the inequality (\ref{bh1}) for $p=\infty$ is obtained from (\ref{ya}) choosing $m=2$ and $t=1/2$.

%{\color{red}
%We would like to mention that inequality (\ref{ya}) is not true for unitarily invariant norm in general. We can find many counterexample for the case $t=1/3$ and $m=2$. Indeed, for the following matrices
%$$A = \left(\begin{array}{ccc}
%0.3748 & 1.6499 & 0.2891\\
%1.6499 & 9.2514 & -2.1025\\
%0.2891 &  -2.1025& 6.0845 \\ 
%\end{array}
%\right),\quad
%B = \left(\begin{array}{ccc}
%3.1949 & 0.1509 & 1.7979\\
%0.1509 & 3.1907 & -0.3156\\
%1.7979&-0.3156& 1.9009 
%\end{array}
%\right).
%$$
%the vector of eigenvalues of $P_{1/3}(A, B)$ and $Q_{1/3}(A, B)$ are $(38.1936, 0.3163, 0.2538)$ and $(6.0173, 3.6026, 0.2785)$ respectively. So, Conjecture 2 in \cite{bh} is false.
% }

Recall that the family of Heron mean \cite{heronb} for nonnegative number $a, b$ is defined as 
\begin{equation}\label{heron}
H_t(a, b) = (1-t)(\frac{a+b}{2}) + t\sqrt{ab},\quad 0 \le t \le 1.
\end{equation}
The Kubo-Ando extension of this to matrices is 
$$(1-t) \frac{A+B}{2} + t(A \sharp B)$$ 
that connects the arithmetic mean and the geometric mean, and a naive extension is 
$$
(1-t) \frac{A+B}{2} + t\frac{A^{1/2}B^{1/2}+B^{1/2}A^{1/2}}{2}
$$ that connects the arithmetic mean and the midpoint of the Heinz mean $\frac{A^{s}B^{1-s}+B^{s}A^{1-s}}{2}.$ So, inequality (\ref{bh1}) is a special case of the following
$$
||(1-t) \frac{A+B}{2} + t(A \sharp B)||_p \le  ||(1-t) \frac{A+B}{2} + t\frac{A^{1/2}B^{1/2}+B^{1/2}A^{1/2}}{2}||_p
$$
with $t= 1/2.$

Notice that another naive extension of the Heron mean for positive definite matrices $A$ and $B$ is defined as
$$
(1-t) \frac{A+B}{2} + t\frac{A^{t}B^{1-t}+A^{1-t}B^{t}}{2}.
$$

In this paper, we extend the inequality (\ref{bhatia1}) to $t$-geometric means. More precisely, we prove that for any positive definite matrices $A, B$ and for any $t\in [0, 1]$
\begin{equation*}\label{111}
\lambda (A^{1/2} (A\sharp_t B) A^{1/2}) \prec_{\log} \lambda (A^{1-t} B^t A^{1-t}).
\end{equation*}

Using this extension, we prove the following result:
\begin{theorem}\label{main1} Let $A, B$ be positive definite matrices and $p=1, 2$ and $r\ge 0$. Then 
\begin{equation}\label{th2}
||A+ B + r(A\sharp_t B+A\sharp_{1-t} B)||_p \le ||A+ B + r(A^{t}B^{1-t} + A^{1-t}B^t)||_p.
\end{equation}
\end{theorem}

Also, using the approach in \cite{BEL} we show that 
for positive definite matrices $A, B$ and for any $z$ in the strips $S_{1/4}=\{z \in \mathbb{C}: \hbox{Re}(z) \in [1/4, 3/4]\}$, 
\begin{equation*}\label{main}
|\Tr(A^{1/2}B^z A^{1/2}B^{1-z})| \le \Tr(AB).
\end{equation*} 

\section{Inequalities}

\begin{proposition}\label{mainprop}
Let $A, B$ be positive definite matrices. Then for any $t\in [0, 1]$
\begin{equation}\label{111}
\lambda (A^{1/2} (A\sharp_t B) A^{1/2}) \prec_{\log} \lambda (A^{1-t/2} B^t A^{1-t/2}).
\end{equation}
\end{proposition}

\begin{proof}
Firstly, let's prove 
\begin{equation}\label{lem1}
\lambda_1 (A^{1/2} (A\sharp_t B) A^{1/2}) \le \lambda_1 (A^{1-t/2} B^t A^{1-t/2}).
\end{equation}
This inequality is equivalent to the statement 
$$
A^{1-t/2} B^t A^{1-t/2} \le I \quad \Longrightarrow \quad A^{1/2} (A\sharp_t B) A^{1/2} \le I,
$$
which in turn is equivalent to 
\begin{equation}\label{222}
B^t \le A^{t-2} \quad \Longrightarrow \quad (A^{-1/2}BA^{-1/2})^t \le A^{-2}.
\end{equation}
That can be proved by using the Furuta inequality which states that if $0 \le Y \le X$, then for all $p\ge 1$ and $r\ge 0$ we have
\begin{equation}\label{lem2}
(X^{r}Y^pX^r)^{1/2} \le (X^{p+2r})^{1/p}.
\end{equation}
Let apply (\ref{lem2}) to $X=A^{t-2}$, $Y=B^t$, $p=\frac{1}{t}$ and $r=-\frac{1}{2(t-2)}$, we get (\ref{222}), and hence (\ref{lem1}). 

Denote by $C_k(X)$ the $k$-th compound of $X\in M_n$, $k=1, \dots, n$. Note that for any positive definite matrices $X, Y$, 
 \begin{eqnarray}
 C_k(A^{1/2}(A\sharp_t B)A^{1/2}) &=& C_k(A(A^{-1/2}BA^{-1/2})^tA)  \nonumber \\
 &=& C_k(A)C_k((A^{-1/2}BA^{-1/2})^{t})C_k(A)  \nonumber \\
& =& C_k^{1/2}(A)(C_k^{1/2}(A)(C_k^{-1/2}(A)C_k(B)C_k^{-1/2}(A))^{t}C_k^{1/2}(A))C_k^{1/2}(A) \nonumber \\
& =& C_k^{1/2}(A)(C_k(A)\sharp_t C_k(B))C_k^{1/2}(A). \label{compound}
 \end{eqnarray}
In the other hand, 
 \begin{equation}\label{lk}
\lambda_1(C_k(A^{1/2}(A\sharp_t B)A^{1/2}))= \prod_{i=1}^k \lambda_i (A^{1/2}(A\sharp_t B)A^{1/2}), \quad k=1, \dots, n-1,.
\end{equation} 
 On account of (\ref{lem1}) and (\ref{lk}) for $1\le k \le n$ we have
 \begin{align*}
\prod_{i=1}^k \lambda_i (A^{1/2}(A\sharp_t B)A^{1/2}) & = \lambda_1(C_k(A^{1/2}(A\sharp_t B)A^{1/2})) \\
 & \le \lambda_1 (C_k(A)^{1-t} C_k(B)^t C_k(A)^{1-t}) \\
 &= \prod_{i=1}^k \lambda_i(A^{1-t} B^t A^{1-t}).
 \end{align*}
The equality holds for $k=n$, since  $ \det (A^{1/2}(A\sharp_t B)A^{1/2)} =\det(A^{1-t} B^t A^{1-t}).
$

Thus, we have proved (\ref{111}). 
\end{proof}

The following special case of Proposition  will be used in the proof of the main result.
\begin{corollary}\label{cormain}
For any positive definite matrices $A$ and $B,$
$$
\Tr(A(A\sharp_t B)) \le \Tr(A^{2-t} B^t),\quad t\in [0,1].
$$
\end{corollary}

In order to prove the next result, let's recall the generalized H\o lder inequality for trace \cite[Theorem 2.8]{simon}: let $\frac{1}{p}+\frac{1}{q}+\frac{1}{r}=1$ for $p, q, r \ge 1$ and $X, Y, Z$ be matrices in $M_n$, then 
$$
\Tr(XYZ) \le ||XYZ||_1 \le ||X||_p||Y||_q||Z||_r.
$$
We also need the famous Lieb-Thirring inequality: $\Tr((AB)^m) \le \Tr(A^mB^m)$.
\begin{theorem}\label{mainlemma}
Let $X, Y$ be positive definite matrices and $z \in S_{1/4}=\{z \in \mathbb{C}: \hbox{Re}(z) \in [\frac{1}{4}, \frac{3}{4}]\}$. Then
\begin{equation}\label{main}
|\Tr(X^{1/2}Y^z Y^{1/2}Y^{1-z})| \le \Tr(XY).
\end{equation}  
\end{theorem}
\begin{proof}
Let $z = \frac{1}{2} + iy, y\in \mathbb{R}$ denote any point in the vertical line of the complex plane passing $x= 1/2$. Then we have
\begin{align*}
|\Tr(X^{1/2}Y^z X^{1/2}Y^{1-z})|  & = |\Tr(X^{1/2}Y^{1/2}Y^{iy} X^{1/2}Y^{1/2}Y^{-iy})| \\
& \le \Tr(|X^{1/2}Y^{1/2}Y^{iy} X^{1/2}Y^{1/2}Y^{-iy}|) \\
& \le  || X^{1/2}Y^{1/2}Y^{iy}||_2 ||X^{1/2}Y^{1/2}Y^{-iy}||_2 \\
& = ||X^{1/2}Y^{1/2}||^2_2 \\
& = \Tr(XY).
\end{align*}
The first inequality is obvious, the second one follows from the Cauchy-Schwarz inequality for trace, and the second equality is from the fact that $Y^{iy}$ and $Y^{-iy}$ are unitary operators.

Now let consider $z= \frac{1}{4} + iy, y\in \mathbb{R}$, a generic point in the vertical line over $x=1/4$, then by using the H$\ddot{o}$lder inequality with $\frac{1}{4}+\frac{1}{4}+\frac{1}{2}=1$ and the Araki-Lieb-Thirring inequality we have
\begin{align*}
|\Tr(X^{1/2}Y^z X^{1/2}Y^{1-z})|  & = |\Tr(X^{1/2}Y^{1/4}Y^{iy} X^{1/2}Y^{1/2}Y^{-iy}Y^{1/4})| \\
& =|\Tr(Y^{1/4}X^{1/4}X^{1/4}Y^{1/4}Y^{iy} X^{1/2}Y^{1/2}Y^{-iy})| \\
%& \le \Tr(|Y^{1/4}X^{1/2}Y^{1/4}Y^{iy} X^{1/2}Y^{1/2}Y^{-iy})|) \\
& \le  || Y^{1/4}X^{1/4}||_{1/4}^2 ||X^{1/2}Y^{1/2}||_2 \\
& \le || Y^{1/2}X^{1/2}||_2 ||X^{1/2}Y^{1/2}||_2 \\
& = \Tr(XY).
\end{align*}
Mention that the map $x \mapsto A^z = e^{x\ln A} = \sum_k z^k \frac{(\ln A)^k}{k!}$ is analytic for $A > 0$, the product of matrices is also analytic and the trace is complex linear, the function $f(z)=\Tr(X^{1/2}Y^zX^{1/2} Y^{1-z})
$
is entire. Moreover, by the similar above argument for $z=x+iy$ it is easy to see that if $0 \le x \le 1$ then the function is bounded. By the Hadamard three-lines theorem the supremum $M(x) = \sup \{|f(x+iy)|: y \in \mathbb{R}\}$ of the function $\Tr(X^{1/2}Y^zX^{1/2} Y^{1-z})$ is log-convex, that means, for any $\lambda \in [0, 1],$
$$
M(\lambda x_1 + (1-\lambda)x_2) \le M(x_1)^\lambda M(x_2)^{1-\lambda} \le \Tr(XY)^\lambda \Tr(XY)^{1-\lambda} = \Tr(XY).
$$
Therefore, the bound $\Tr(XY)$ is valid in the vertical strip $1/4 \le \hbox{Re}(z) \le 1/2.$ Invoking the symmetry $z\mapsto 1-z$ and exchanging the roles of $A$ and $B$ give the desired bound on the full strip $S_{1/4} = \{1/4 \le \hbox{Re}(z) \le 3/4\}$.
\end{proof}
As a consequence, we have the following inequality (see \cite[Inequality (39)]{bh}): 
\begin{corollary}\label{cor} For any positive definite matrices and for $t\in [0,1]$.
$$
\Tr((A\sharp_t B)(A\sharp_{1-t} B)) \le \Tr(AB).
$$
\end{corollary}

Now we are ready to prove the main result in this paper.
\begin{theorem}\label{theorem} Let $A, B$ be positive definite matrices, $p=1, 2$ and $r\ge 0$. Then 
\begin{equation}\label{th2}
||A+ B + r(A\sharp_t B+A\sharp_{1-t} B)||_p \le ||A+ B + r(A^{t}B^{1-t} + A^{1-t}B^t)||_p.
\end{equation}
\end{theorem}
\begin{proof}
Since $A+ B + r(A\sharp_t B+A\sharp_{1-t} B)\ge 0$, the left hand side of (\ref{th2}) is $\Tr(A+ B + r(A\sharp_t B+A\sharp_{1-t} B))$. It is well-known that $\Tr(A\sharp_t B) \le \Tr(A^{1-t}B^t)$ and $\Tr(A\sharp_{1-t} B) \le \Tr(A^{t}B^{1-t})$. We have  
\begin{align*}
\Tr(A+ B + r(A\sharp_t B+A\sharp_{1-t} B)) & \le \Tr(A+ B + r(A^{t}B^{1-t} + A^{1-t}B^t)) \\
& \le \Tr(|A+ B + r(A^{t}B^{1-t} + A^{1-t}B^t)|).
\end{align*}
So for $p=1$ the inequality (\ref{th2}) follows. 

Next consider the case $p=2$. Notice again that $\Tr((A\sharp_t B)^2) \le \Tr(B^{2t}A^{2(1-t)})$ (see \cite[pape 121]{bh}). Similarly, we also have $\Tr( (A\sharp_{1-t} B)^2) \le \Tr(A^{2t}B^{2(1-t)})$. Then 
\begin{equation}\label{21}
\Tr( (A\sharp_t B)^2+(A\sharp_{1-t} B)^2) \le \Tr(A^{2t}B^{2(1-t)} + B^{2t}A^{2(1-t)}).
\end{equation}
By Proposition \ref{mainprop} we have 
\begin{equation}\label{22}
\Tr( (A+B)(A\sharp_t B+A\sharp_{1-t} B)) \le \Tr(A^{t+1} B^{1-t}+A^{2-t}B^t + A^tB^{2-t} + A^{1-t}B^{1+t}).
\end{equation}
Now, squaring both sides of (\ref{th2}), we need to show 
\begin{align*}
&  \Tr((A+ B)^2 + r^2(A\sharp_t B)^2+r^2(A\sharp_{1-t} B)^2+ 2r(A+B)(A\sharp_t B+A\sharp_{1-t} B) + 2r^2(A\sharp_t B)(A\sharp_{1-t} B))\\
&\le  \Tr((A+ B)^2 + 2r(A^{t+1} B^{1-t}+A^{2-t}B^t + A^tB^{2-t} + A^{1-t}B^{1+t}) + r^2A^{2t}B^{2(1-t)} )+r^2B^{2t}A^{2(1-t)} \\
&\quad + 2r^2\Tr(AB). 
\end{align*}
The last inequality follows from (\ref{21}), (\ref{22}) and Corollary \ref{cor}.
\end{proof}
\begin{remark}
From Theorem \ref{theorem} for $s \in [0, 1]$ we 
$$
||(1-s)(A+B) + s(A\sharp_t B + A\sharp_{1-t} B)||_p \le ||(1-s)(A+ B) + s(A^{t}B^{1-t} + A^{1-t}B^t)||_p.
$$
When $t=1/2$ we obtain one kind of inequality for the matrix Heron mean
$$
||\frac{1-s}{2}(A+B) + s(A\sharp B) ||_p \le ||\frac{1-s}{2}(A+ B) + sA^{1/2}B^{1/2}||_p.
$$
\end{remark}
\begin{remark} By the same arguments, one can show that 
$$
||A+B + A\sharp_t B + A\sharp_{1-t} B ||_p \le ||A+ B + A^{t}B^{1-t} + A^{1-t}B^t||_p.
$$
But is we use another version of the Heinz mean $(A^{t}B^{1-t} + B^tA^{1-t})/2$ and realize the same proof in Theorem \ref{theorem} the inequality in Corollary \ref{cor} could be as follows
\begin{equation}\label{th122}
\Tr((A \sharp_t B)(B\sharp_t A))  \le \hbox{Re} \Tr( A^tB^t A^{1-t}B^{1-t}).
\end{equation}
Notice that both sides are bounded by $\Tr(AB)$ but it is not clear that (\ref{th122}) is true or not.
\end{remark}

From the proof of the main theorem, it is natural to ask the following question:  Is it true that for $0 \le X, Y \le Z$ such that $X \prec_{\log} Y$
\begin{equation}
Z^{1/2} X Z^{1/2} \prec_{\log}  Z^{1/2} Y Z^{1/2}?
\end{equation}

Unfortunately, the answer is negative. Indeed, let $$X = \begin{pmatrix}1&0\\0&1 \end{pmatrix}, \quad
Y = \begin{pmatrix}
0 &1\\1&0
\end{pmatrix},\quad 
Z= \begin{pmatrix}
1 &0\\0&4
\end{pmatrix}.
$$
Now $s(X) =s(Y)  =(1,1)$ but  $s(Z^{1/2}XZ^{1/2}) = (1, 2)  \not\prec_w (\sqrt 2,\sqrt 2) =s(Z^{1/2}YZ^{1/2})$. So it is not even true for diagonal positive definite matrices.

\section{Determinant Inequality for the Heron mean}
Let's recall a recent result of Audeanert \cite{AD}: for any positive semidefinite matrices $A$ and $B$ 
\begin{equation}
\hbox{det}(I + A\sharp B) \le \hbox{det}(I + A^{1/2} B^{1/2}). 
\end{equation}
The author used the well-known fact that $\lambda (A\sharp B) \prec_{\log} \lambda (A^{1/2}B^{1/2})$ and the function $\Phi(X)= \sum_{i=1}^n \log(1+ e^{x_i})$ ($X = (x_1, x_2, \cdots, x_n)$)  is isotone (i.e. the function preserving weak majorization: $x \prec y \Rightarrow  \Phi(x) \prec_w \Phi(y)$.)

In fact, for matrices $A$ and $B$ such that $\lambda(A) \prec_{\log} \lambda(B)$ we have 
\begin{equation}\label{au1}
\Dt(I + A) \le \Dt(I+B).
\end{equation}
A useful characterization of isotone functions in the case $m=1$ is as follows:

\begin{lemma}
A differentiable function $\Phi: \mathbb{R}^n \to \mathbb{R}$ is isotope if and only if it satisfy
\begin{itemize}
\item[(1)] $\Phi$ is permutation invariant;
\item[(2)] for all $X \in \mathbb{R}^n$ and for all $i, j$:
$$
(x_i-x_j) \big(\frac{\partial \Phi}{\partial x_i} (x)- \frac{\partial \Phi}{\partial x_j} (x)\big) \ge 0.
$$
\end{itemize}
\end{lemma} 
Do the similar argument as in \cite{AD} one can prove the following
\begin{equation}
\Dt (I + A\sharp_t B) \le \Dt (I + A^{1-t} B^{t}). 
\end{equation}

Now we can use this fact to obtain some inequality for the Heron mean.

\begin{theorem}
For any positive definite matrices $A$ and $B$
\begin{equation}\label{det}
\Dt (P_{t}(A, B)) \le \Dt (Q_{t}(A, B)).
\end{equation}
\end{theorem}
\begin{proof}
The inequality (\ref{det}) is equivalent to the following 
\begin{align*}
 \Dt^{1/t} (A^{t}+B^{t}) & = \Dt(A) \Dt^{1/t} (I+A^{-t/2}B^{t}A^{-t/2})\\
  & \ge  \Dt(A\sharp_{1/t}(A+A\sharp_t B)) \nonumber \\ 
 & =\Dt(A) \cdot \Dt^{1/t} (I + (A^{-1/2}BA^{-1/2})^t)
\end{align*} 
or
\begin{align}\label{11111}
\Dt (I+A^{-t/2}B^{t}A^{-t/2}) \ge \Dt (I + (A^{-1/2}BA^{-1/2})^t).
\end{align} 
By the Araki-Lieb-Thirring inequality we have
$$
\lambda(I + (A^{-1/2}BA^{-1/2})^t) \prec_{\log} \lambda(I+A^{-t/2}B^{t}A^{-t/2}).
$$
Therefore, the inequality (\ref{11111}) follows from the last inequality and (\ref{au1}). 
\end{proof}
As a consequence, we obtain a determinant inequality for the Heron mean.
\begin{corollary}
For any positive definite matrices $A$ and $B,$
$$
\Dt(A+ B + 2(A\sharp B)) \le \Dt(A+ B + A^{1/2}B^{1/2} + A^{1/2}B^{1/2})).
$$
\end{corollary}

%{\color{red}
%Let's get back to Conjecture 1 in \cite{bh}. If we can prove that
%\begin{equation}
%2(A+B)^{-1/2} (A\sharp B) (A+B)^{-1/2}\prec_{\log} (A+B)^{-1/2}(A^{1/2}B^{1/2}+B^{1/2}A^{1/2})(A+B)^{-1/2},
%\end{equation}
%then we can hope to use the concept of isotone functions to get
%\begin{equation}
%I + 2(A+B)^{-1/2} (A\sharp B) (A+B)^{-1/2}\prec_{\log} I + (A+B)^{-1/2}(A^{1/2}B^{1/2}+B^{1/2}A^{1/2})(A+B)^{-1/2}
%\end{equation}
%that is equivalent to
%$$|| I + 2(A+B)^{-1/2} (A\sharp B) (A+B)^{-1/2} || \leq  ||I + (A+B)^{-1/2}(A^{1/2}B^{1/2}+B^{1/2}A^{1/2})(A+B)^{-1/2} ||
%$$
%}
%This cries out for an obvious generalization for any number of matrices which we state as:
%{\bf Conjecture.} For any matric

\end{document}